\documentclass[12 pt, a4paper]{article}

\usepackage{amsfonts}

\usepackage{epsfig}
\usepackage{psfrag}

\newcommand{\fracs}[2]{{ \textstyle \frac{#1}{#2} }}
\newtheorem{theorem}{Theorem}[section]

\newtheorem{corollary}{Corollary}[section]

\newtheorem{remark}{Remark}[section]
\newenvironment{proof}{
 \bgroup\noindent\small{\bf Proof\ }}{
 \nolinebreak\hbox{\ $\Box$}
 \egroup}

\newcommand{\dW}{\, {\rm d} W}
\newcommand{\dM}{\, {\rm d} M}
\newcommand{\dt}{\, {\rm d} t}
\newcommand{\ds}{\, {\rm d} s}
\newcommand{\dX}{\, {\rm d} X}
\newcommand{\dx}{\, {\rm d} x}

\newcommand{\dy}{\, {\rm d} y}
\newcommand{\dv}{\, {\rm d} v}
\newcommand{\dw}{\, {\rm d} w}
\newcommand{\e}{\, {\rm e}}

\newcommand{\Ito}{It{\^o} }

\newcommand{\EE}{\mathbb{E}}
\newcommand{\VV}{\mathbb{V}}
\newcommand{\RR}{\mathbb{R}}
\newcommand{\ZZ}{\mathbb{Z}}

\begin{document}

\title{
Mean-square stability and error analysis of
implicit time-stepping schemes for linear parabolic SPDEs with multiplicative Wiener noise in the first derivative}
\author{Christoph Reisinger\thanks{Mathematical Institute
        and Oxford-Man Institute of Quantitative Finance, University of Oxford,
        {\tt christoph.reisinger@maths.ox.ac.uk}}}

\maketitle

\begin{abstract}
In this article, we extend a Milstein finite
difference scheme introduced in \cite{gr11}
for a certain linear stochastic partial differential equation (SPDE),
to semi- and fully implicit timestepping as introduced by \cite{szpruch} for SDEs.
We combine standard finite difference Fourier analysis for PDEs
with the linear stability analysis in \cite{buckwar} for SDEs, to analyse the stability and accuracy.
The results show that Crank-Nicolson time-stepping for the principal part of the drift with a partially implicit but negatively
weighted double \Ito integral gives unconditional stability over all parameter values, and converges
with the expected order in the mean-square sense.
This opens up the possibility of local mesh refinement in the
spatial domain, and we show experimentally that this can be beneficial in the presence of reduced regularity at
boundaries.
\end{abstract}

\noindent
{\small
{\bf Keywords}: Stochastic partial differential equations, finite differences, implicit timestepping schemes, Fourier analysis, local mesh refinement}

\section{Introduction}


The numerical analysis and computation of stochastic partial differential equations (SPDEs) have become a subject of active research over the recent past.
The interest has been triggered partly by applications in areas as diverse as geophysics \cite{wt02} and mathematical finance \cite{heath92,mz09},
and has led to questions regarding the complexity theory of their 
approximation \cite{mgr07, mgrw07}.

In a prominent class of SPDEs, the stochasticity enters 
via a random driver of the form
\begin{equation}
\label{stdspde}
\dv = (A v + f(v)) \dt + g(v) \dM_t,
\end{equation}
where $A$ is a linear elliptic operator, e.g.\ the Laplace operator, $M$ a martingale driver, often standard Brownian motion, and
$f$ and $g$ non-linear functions, e.g.\ with Lipschitz regularity. This leads to special cases with additive or multiplicative noise terms.
We will consider here a variant of equation (\ref{stdspde}).


Typical solutions are by lattice methods, e.g.\
\cite{gyongy, gyongy97},
finite differences, e.g.\ \cite{roth02}, or by finite elements, see e.g.\ \cite{walsh},
with extensions to higher order Taylor schemes
\cite{jk09, jk10, jetal11}, as well as
multilevel schemes \cite{barth2}.



In the following, let $(\Omega,\mathcal{F},P)$ a probability space,
$M$ a one-dimensional standard Brownian motion adapted to $\mathcal{F}$.
We study specifically the equation
\begin{equation}
\label{spde}
\dv = -\mu\, \frac{\partial v}{\partial x} \dt 
+ \frac{1}{2} \, \frac{\partial^2  v }{\partial x^2} \dt
- \sqrt{\rho}\, \frac{\partial v}{\partial x} \dM_t,
\end{equation}
where $\mu$ and $0\le \rho< 1$ are real-valued parameters.
It is a classical result from \cite{krylov81} that for a class of parabolic SPDEs including
(\ref{spde}), with initial data in $H^1$, there is a unique weak solution $v \in L_2(\Omega\times (0,T),\mathcal{F}, H^1(\mathbb{R}))$.
In fact, for the special form (\ref{spde}) on $\mathbb{R}$, i.e.\ without boundaries, it is easy to see that
a solution is given by
\[
v(t,x) = u(t,x-\mu t-\sqrt{\rho} M_t),
\]
where $u$ is the solution to the heat equation
\[
\frac{\partial u}{\partial t} = \fracs{1}{2} (1-\rho) \frac{\partial^2  u }{\partial x^2}
\]
with the same initial data as (\ref{spde}). We will use this semi-analytical solution to measure the errors of numerical approximations.

We list two applications of this equation.
\cite{kx99} show that (\ref{spde}) describes the limit empirical measure of a large particle system, where each individual
particle is governed by
\begin{equation}
\label{iproc}
\dX_t^i = \mu \dt + \sqrt{1-\rho} \dW_t^i + \sqrt{\rho} \dM_t,
\end{equation}
where $W^i$ are standard Brownian motions, which are independent mutually and of the Brownian motion $M$.
The parameter $\rho$ describes the correlation between each pair of $X^i$, which
explains the motivation for choosing $0\le \rho< 1$ in (\ref{spde}).
We will see later that $\rho=1$ is a boundary case in the stability analysis. It is also clear how equations with
non-normalised constant coefficients can be rescaled to (\ref{spde}).

Equation (\ref{spde}) also arises as the Zakai equation in a stochastic filtering problem (see, e.g.\ \cite{crisan}),
 where the solution is the distribution of a signal $X$, based on noisy observation of $M$.

\cite{gr11} introduce a Milstein finite difference approximation of (\ref{spde}) and study the complexity of multi-level Monte Carlo simulation.
In this article, we extend the discretisation and its analysis to an implicit method on the basis of the $\sigma$-$\theta$ time-stepping scheme proposed
and analysed by \cite{szpruch, buckwar} for SDEs, where the drift and the deterministic part of the double stochastic integral are taken (partly) implicit.
By combining Fourier methods in \cite{buckwar} and \cite{gr11}, we obtain a stability condition on the ratio $k/h^2$, under which the approximations
to the initial-value problem of (\ref{spde}) converge in mean-square sense in the spatial $L_2$- and $L_\infty$-norms, of first order in the time-step $k$ and second order in the spatial mesh size $h$.
A peculiarity of equation (\ref{spde}) is that the stability region of the chosen scheme is larger for explicit treatment of
the Milstein correction than for partly implicit treatment as explained above, and that an `anti'-implicit version with negative weight of the implicit term gives unconditional
stability.
We find, both from an asymptotic expansion of the error and numerical experiments, that the numerical error is dominated by the stochastic terms of the equation and therefore implicit or even Crank-Nicolson-type versions of the scheme have little effect on the achieved accuracy.
The ratio of $k/h^2$, which gives, empirically, optimal accuracy for constant mesh sizes, is close to the stability limit of the explicit scheme.
The improved stability can, however, be useful for locally refined schemes. 


The rest of the article is organised as follows.
In Section \ref{sec:finite-difference}, we define the implicit Milstein finite difference schemes and analyse their stability and accuracy
by Fourier techniques. Section \ref{sec:firsttests} presents numerical tests which confirm and illustrate these findings.
Section \ref{sec:results} gives an application to the pricing of basket credit derivatives, where the presence of an absorbing boundary
leads to local loss of regularity, and we show how mesh grading in conjunction with unconditionally stable implicit schemes can be used to improve the accuracy.
Section \ref{sec:conclusions} gives conclusions and outlines directions for future research.

\section{Discretisation and analysis of stability and convergence}
\label{sec:finite-difference}

Starting point is the integrated form of the SPDE (\ref{spde}), over a time interval $[t,t\!+\!k]$,
\[
v(t\!+\!k,x) = v(t,x) + \int_t^{t+k} \left( -\,  \mu \frac{\partial v}{\partial x}
                + \frac{1}{2}\, \frac{\partial^2  v }{\partial x^2}\right) \ds - 
          \int_t^{t+k} \!\!\sqrt{\rho}\ \frac{\partial v}{\partial x}\,  \dM_s.
\]

\subsection{Milstein finite differences}

In \cite{gr11}, a Milstein approximation to the stochastic integral is used, together with
standard central difference approximations 
on
a spatial grid with uniform spacing $h$,
to obtain an approximation $v_j^n$ to $v(nk,jh)$ defined by
\begin{eqnarray}
v_j^{n+1} &=& v_j^n\ -\ \frac{\mu\, k + \sqrt{\rho\, k}\, Z_n}{2h} \left(v_{j+1}^n - v_{j-1}^n\right) 
\nonumber \\&&~~~ +\ \frac{(1\!-\!\rho)\, k + \rho \, k\, Z_n^2}{2h^2} 
 \left(v_{j+1}^n - 2 v_j^n + v_{j-1}^n\right),
\label{discrete}
\end{eqnarray}
where $Z_n \sim N(0,1)$ are independent, for $n\ge 0$.

For a vector $V_n = (\ldots, v_{-1}^n,v_0^n,v_1^n,v_2^n,\ldots) \in \RR^\ZZ$, the system can then be written in operator form
\begin{eqnarray}
V_{n+1} &=& V_n\ -\ \frac{\mu\, k + \sqrt{\rho\, k}\, Z_n}{2h} D_1 V_n
+\ \frac{(1\!-\!\rho)\, k + \rho \, k\, Z_n^2}{2 h^2}  D_2 V_n,
\label{discrete-sys}
\end{eqnarray}
where $D_1$ and $D_2$ are first and second central difference operators.

\begin{remark}
\label{remark-mol}
The discretisation arises from a `horizontal' method of lines, where the time integral is approximated first, and then the spatial
derivatives are approximated by finite differences.
The `vertical' version where the Milstein scheme is applied to the system of SDEs resulting from a finite difference approximation
of the spatial derivatives, leads to 
\begin{eqnarray}
V_{n+1} = V_n - \frac{\mu\, k + \sqrt{\rho\, k}\, Z_n}{2h} D_1 V_n
+  \frac{k}{2 h^2}  D_2 V_n
+  \frac{\rho \, k\, (Z_n^2-1)}{2 h^2}  D_1^2 V_n.
\label{discrete-sys2}
\end{eqnarray}
The only difference is in the \Ito term, where the second difference is replaced by an iterated first difference.
We will sketch in Remark \ref{remark1} why the properties of the schemes are asymptotically identical, while the scheme proposed earlier has
implementational advantages as it leads to more compact finite difference stencils.
\end{remark}

%

\cite{gr11} derive the condition $(1+2 \rho^2) k/h^2 \le 1$ for mean-square stability of this explicit scheme. For $\rho=0$, this reduces to the
well known stability condition for the standard heat equation. The limitation on the timestep
is the motivation for
considering the following implicit versions.

In the spirit of \cite{kp92}, pp.~399, we define an \emph{implicit Milstein} 
finite difference scheme by
\begin{eqnarray}
\nonumber
V_{n+1} &=& V_n\ -\ \frac{\mu\, k}{2h} D_1 V_{n+1}
+\ \frac{k}{2 h^2}  D_2 V_{n+1}  \\
&& ~~~~ -\ \frac{\sqrt{\rho\, k}\, Z_n}{2h} D_1 V_n
+\ \frac{\rho \, k\, (Z_n^2-1)}{2 h^2}  D_2 V_n.
\label{discretesemiimpl}
\end{eqnarray}
All drift terms are taken implicit, while the volatility terms are taken explicit.
We also define a $\theta$-scheme
\begin{eqnarray}
\nonumber
V_{n+1} &=& V_n\ -\ \theta \left(\frac{\mu\, k}{2h} D_1 - \frac{k}{2 h^2}  D_2\right) V_{n+1}
- (1-\theta) \left(\frac{\mu\, k}{2h} D_1 - \frac{k}{2 h^2}  D_2\right) V_{n}
 \\
&& ~~~~ -\ \frac{\sqrt{\rho\, k}\, Z_n}{2h} D_1 V_n
+\ \frac{\rho \, k\, (Z_n^2-1)}{2 h^2}  D_2 V_n,
\label{discretetheta}
\end{eqnarray}
for $\theta\in [0,1]$.
Clearly, for $\theta=0$ one recovers the explicit scheme, for $\theta=1$ the implicit scheme.

It is pointed out in \cite{higham2} that
the stability region of drift-implicit Milstein schemes is often lower than their Euler-Maruyama counterpart.
\cite{szpruch} and \cite{buckwar} suggest a  $\sigma$-$\theta$-scheme,
which translates into the present SPDE setting as
\begin{eqnarray}
\nonumber
V_{n+1} &=& V_n\ -\ 
\theta \left(\frac{\mu\, k}{2h} D_1 - \frac{k}{2 h^2}  D_2\right) V_{n+1}
- (1-\theta) \left(\frac{\mu\, k}{2h} D_1 - \frac{k}{2 h^2}  D_2\right) V_{n}
 \\
 \nonumber
&& ~~~~ -\ \sigma \frac{\rho \, k}{2 h^2}  D_2 V_{n+1} - (1-\sigma) \frac{\rho \, k}{2 h^2}  D_2 V_n 
 \\
&& ~~~~ -\ \frac{\sqrt{\rho\, k}\, Z_n}{2h} D_1 V_n
+\ \frac{\rho \, k\, Z_n^2}{2 h^2}  D_2 V_n,
\label{discreteimpl}
\end{eqnarray}
where the deterministic part of the double \Ito integral is also taken partly implicit with $\sigma\in [0,1]$.
Note that all terms that can be taken implicit, consistent with the \Ito integral, are taken implicit.
Implicitness of terms involving $M$ changes the character of the integral, e.g.\ in the Stratonovic sense
for a trapezium rule approximation.

We now analyse accuracy and stability of the above schemes. The analysis is done on the real line (infinite grid),
for analytical tractability,
although in practical
applications truncation to a finite domain and approximation on a finite grid will be necessary.
We outline this at the start of Section \ref{sec:firsttests} and discuss the boundary behaviour in Section \ref{sec:results}.

\subsection{Mean-square stability analysis of Fourier modes}
\label{subsec:stab}

We assume for simplicity $\mu=0$ in the following, but the results are unaltered in the case $\mu\neq 0$, as we will discuss briefly 
in Remark \ref{remark2}.

As per classical finite difference analysis, e.g.\ \cite{rm67,mm05},
we study simple Fourier mode solutions of the form
\begin{equation}
\label{fouriermode}
V_j^n = X_n \exp(i j \phi), \quad |\phi| \leq \pi.
\end{equation}
We use superposition of these solutions for different $\phi$ to construct the leading order error terms
in the next section, and for now focus on the stability of individual modes.

Following \cite{higham, sm96},
we say that the scheme is mean-square stable, if for all $\phi \neq 0$, for the amplitude $X_n$ of the corresponding Fourier mode,
\begin{equation}
\label{ms-stab}
\EE\left[\, |X_n|^2 \right] \rightarrow 0 \quad {\rm for } \;\, n\rightarrow \infty.
\end{equation}

\begin{theorem}
Assume $\rho \in [0,1)$.
The $\theta$-$\sigma$ Milstein central difference scheme (\ref{discrete-sys}) is stable in the mean-square sense (\ref{ms-stab}) for Fourier modes
(\ref{fouriermode}),
provided
\begin{eqnarray}
\frac{k}{h^2} f(\rho;\theta,\sigma) :=
\frac{k}{h^2} \left[1 - 2 (\theta-\rho\sigma-\rho^2) \right]&<& 1.
\label{stab-lim}
\label{theta-sig-stab}
\end{eqnarray}
\end{theorem}

\begin{proof}
Insertion of (\ref{fouriermode}) in (\ref{discreteimpl}) leads to the equation
\begin{eqnarray}
\label{four-scheme}
X_{n+1} &=& X_n + k a \left( \theta X_{n+1} + (1-\theta) X_n\right) 
- \sqrt{k} i c Z_n X_n \\
&& \quad \;\; + \,k \rho a \, Z_n^2 \, X_n - k \rho a \left(\sigma X_{n+1} + (1-\sigma) X_n\right),
\end{eqnarray}
where
\begin{eqnarray}
\label{a}
a &=& - \frac{2}{h^2} \sin^2 \frac{\phi}{2}, \\
c &=& \frac{\sqrt{\rho}}{h} \sin\phi.
\label{c}
\end{eqnarray}
Rearranging, taking moduli and expectations gives
\begin{eqnarray}
\label{meansq}
\left(1\!-\!k a (\theta\!-\!\rho\sigma) \right)^2 \EE[\, |X_{n+1}|^2 ] 
\!=\! \left((1\!+\!ka(1\!-\!\theta\!+\!\rho\sigma))^2 \!+\! k c^2 \!+\! 2 k^2 \rho^2 a^2 \right) \EE\left[\, |X_n|^2 \right].
\end{eqnarray}

Simple calculations show that the scheme is stable in the above mean-square sense 
if and only if
\[
1+ k a \left(\frac{1}{2} - \theta + \rho\sigma + \rho^2 \right) > - \frac{c^2}{2 a}.
\]
Re-inserting $a$ and $c$ shows that this is equivalent to
\[
\frac{k}{h^2} \left[1-2 (\theta - \rho\sigma-\rho^2) \right] \sin^2 \frac{\phi}{2} + \rho \cos^2 \frac{\phi}{2} < 1, \quad \forall \phi,
\]
leading to the result.
\end{proof}

\begin{remark}
\label{remark1}
These results are similar but not identical to those obtained for scalar complex-valued test equations in \cite{buckwar},
because the discretisation (\ref{discreteimpl}) differs from the standard Milstein scheme for systems of SDEs, as per Remark \ref{remark-mol}.
The stability conditions only differ by terms which vanish as $k,h\rightarrow 0$ with $k/h^2$ fixed, and are hence asymptotically equivalent. 
This was to be expected given the schemes differ only in the stencil width for the discretisation of the second derivative in the \Ito term.
We do not reproduce the analysis of the other scheme here.
\end{remark}

\begin{remark}
\label{remark2}
Similarly, the inclusion of a first order term in the drift, $\mu\neq 0$, leads to lower order corrections in $k$, $k/h^2$ fixed,
and therefore does not change the result asymptotically.
\end{remark}

The scheme is unconditionally mean-square stable, i.e., without conditions on $k$ and $h$, if $f\le 0$.
For all other cases the scheme is mean-square stable if $k/h^2 < 1/f$.
This upper bound $1/f$ is shown as a function of $\rho$
in Figure \ref{fig:stablim}, for the explicit, implicit and double implicit schemes.

The stability condition (\ref{theta-sig-stab}) is stricter for $\sigma>0$ than for $\sigma=0$, i.e.\ scheme
(\ref{discretesemiimpl}).
Specifically, the doubly implicit Milstein scheme (\ref{discreteimpl}) with $\theta=\sigma=1$, is unconditionally stable in the mean-square sense
only if
$
\rho \le 1/(1+\sqrt{3}),
$
whereas the drift-implicit scheme (i.e., $\sigma=0$, $\theta=1$) is unconditionally stable for $\rho\le 1/\sqrt{2}$.

This arises from the fact that the implict discretisation of the \Ito term on the second line of (\ref{discreteimpl}), 
containing $D_2 V_{n+1}$,
has the opposite sign of the implicit $D_2 V_{n+1}$ term on the first line, which arises from the discretisation
of the $u_{xx}$ term in the SPDE (\ref{spde}). The latter determines the parabolic nature of the problem.
Hence, increasing $\sigma$ reduces this component in the implicit term while it increases it in the explicit term, making
the scheme less contractive in the mean-square sense for all non-zero wave numbers, as is eventually seen from (\ref{meansq}).
Conversely, taking $\sigma<0$ improves the stability, and for $\sigma=-1$, $\theta\ge 1/2$,
the scheme is unconditionally stable for all $0\le \rho< 1$. 
This somewhat surprising feature arises due to the purely imaginary eigenvalues of the first order operator in the Brownian driver.


\begin{figure}[ht]
\begin{center}
\psfrag{rho}{{\small$\rho$}}
\psfrag{limit}{{\small $1/f(\rho;\theta,\sigma)$}}
\psfrag{explicit}{{\scriptsize $\theta=\sigma=0$}}
\psfrag{double implicit}{{\scriptsize $\theta=\sigma=1$}}
\psfrag{implicit}{{\scriptsize $\theta\!=\!1,\sigma\!=\!0$}}
\includegraphics[width=.7\textwidth]{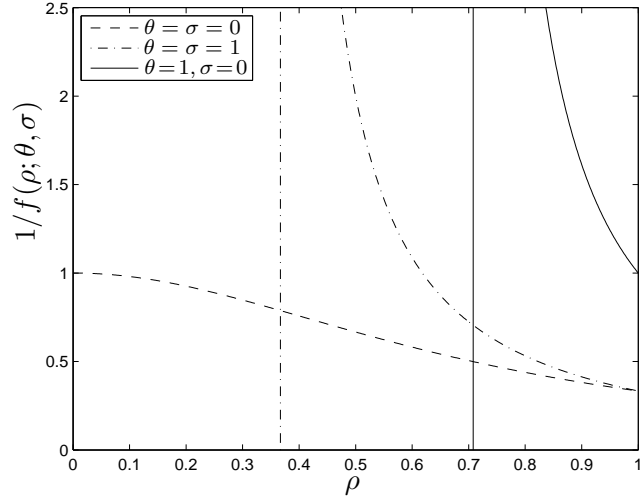}
\end{center}
\caption{
Shown are the stability regions for the 
explicit scheme ($\theta=\sigma=0$),  implicit ($\theta=1$, $\sigma=0$) and double implicit scheme ($\theta=\sigma=1$).
The implicit scheme is unconditionally stable for $\rho\le 1/\sqrt{2} \approx 0.7$,
the double implict scheme for $\rho\le 1/(1+\sqrt{3}) \approx 0.36$, 
marked by vertical lines.
In all other cases, the curve of the style defined in the legend gives the
upper limit of the stability range of $k/h^2$, i.e., $1/f(\rho;\theta,\sigma)$ with $f$ from 
(\ref{theta-sig-stab}), as a function of $\rho$.}
\label{fig:stablim}
\end{figure}

In the case $\theta=1/2$, specifically, the time discretisation of the PDE part is of second order accurate and the error is thus dominated by the Milstein discretisation of the stochastic integral.

\subsection{Fourier analysis of mean-square convergence}

We can also derive the leading order error terms 
exploiting the availability of a closed-form solution.
This is different from the approach in \cite{lang10} who shows a Lax-equivalence theorem for a different class of SPDEs to deduce convergence
from stability and stochastic consistency.

\begin{theorem}
\label{theo:errterm}
Assume $\rho \in [0,1)$, $T>0$, $k=T/N$ and $\frac{k}{h^2}=\lambda>0$ is kept fixed such that (\ref{theta-sig-stab}) holds.
The $\theta$-$\sigma$ Milstein central difference scheme (\ref{discrete-sys}) 
has the error expansion, for Dirac initial data,
\begin{eqnarray}
V_j^N-v(T,j h) = k \, E(T,j h) + o(k) \, R(T,j h),
\end{eqnarray}
where $E$ and $R$ are random variables with bounded moments. 
\end{theorem}
\begin{corollary}
Under the conditions of Theorem \ref{theo:errterm}, the $\theta$-$\sigma$ Milstein scheme converges in the mean-square sense 
for the spatial $L_2$- and $L_\infty$-norms, and
\[
\sqrt{\EE[\|V^N - v(T,\cdot)\|^2]} = O(k).
\]
\end{corollary}

\begin{proof}[of Theorem \ref{theo:errterm}]
By insertion one checks that
$
X(t) \exp( i \kappa x),
$
is a solution to (\ref{spde}) iff
\[
X(t) = X(0) \exp\left(-\fracs{1}{2} (1-\rho) \kappa^2 t - i \kappa \sqrt{\rho} M_t \right).
\]
This allows us to compare the numerical solution from (\ref{four-scheme}), 
\[
X_{n+1} =  X_n \ \frac{1 + k a (1-\theta) - \sqrt{k} i c Z_n - k \rho a ((1-\sigma)-Z_n^2)}{1 - k a \theta +  k \rho a \sigma},
\]
with $a$ and $c$ as in (\ref{a}) and (\ref{c}),
to the exact solution over a single timestep,
\[
X(t_{n+1}) = X(t_n) \ \exp\left( -\fracs{1}{2} (1-\rho) \kappa^2 k - i \kappa \sqrt{\rho\,k} Z_n \right),
\]
where $M_{t_{n+1}}- M_{t_n} = \sqrt{k}\, Z_n$.

We extend here the analysis in \cite{gr11} for the \emph{explicit} scheme, $\theta=\sigma=0$, 
both in scope and in detail.
In the following, we keep $k/h^2=\lambda$ fixed, and set $\phi = h \kappa$.

We first  consider the regime $\kappa \le h^{-m}$, where  $m<1/2$.
Then $x= k \kappa^2 = \lambda h^2 \kappa^2$ is small and one can derive by Taylor expansion
\begin{eqnarray*}
1 + k a A - \sqrt{k} i c Z_n + k \rho a Z_n^2 &=& 1 - x Y_n (1-\fracs{x}{12 \lambda}) - i \sqrt{x} \sqrt{\rho} Z_n (1- \fracs{x}{6 \lambda}) + o(x^2), \\
&=& \exp\left(-\fracs{1}{2} A k \kappa^2 - i \sqrt{\rho k} \kappa Z_n + e_n^{(0)} \right),
\end{eqnarray*}
where $A$ is a fixed constant, $Y_n = \fracs{1}{2} (A+\rho Z_n^2)$ and
\begin{eqnarray*}
e_n^{(0)} =
- i x \sqrt{x \rho} \left[(Y_n -\fracs{1}{6\lambda}) Z_n - \fracs{\rho}{3} Z_n^3\right] +
x^2 \left[(Y_n -\fracs{1}{6\lambda})(\rho Z_n^2-\fracs{1}{2} Y_n) - \fracs{\rho^2}{4} Z_n^4\right] + o(x^2).
\end{eqnarray*}
Similarly, for fixed $B$, 
\begin{eqnarray*}
1 - k a B &=& \exp\left(\fracs{1}{2} B x - e_n^{(1)} \right),\\
e_n^{(1)} &=& \fracs{B}{24 \lambda} x^2 + \fracs{B^2}{8} x^2 + o(x^2).
\end{eqnarray*}
The remainder term $o(x^2) = o(k^2 \kappa^4) = o(k h^2 \kappa^4)$ is understood to be 
a deterministic constant of order $o(x^2)$ multiplied by a
random variable whose moments are all bounded independent of $k$ and $h$.

Taking $A=(1-\theta)-\rho(1-\sigma)$ and $B=\theta-\rho\sigma$, we have $A+B=1-\rho$ and
\[
X_{n+1} =  X_n \ \exp\left( -\fracs{1}{2} (1-\rho) \kappa^2 k - i \kappa \sqrt{\rho\,k} Z_n + e_n \right),
\]
where $e_n = e_n^{(0)} + e_n^{(1)}$. Aggregating over $N$ time steps, at $t_N = k N =T$,
\begin{eqnarray}
X_N &=& X(t_N) \ \exp\left(S_N \right), \\
S_N &=& \sum_{n=0}^{N-1} e_n
\,=\,
k \kappa^4 \mu(\rho,\lambda,A) T \,+\,
i k \kappa^3 \sigma(\rho,\lambda,A) W_T \,+\,
 o(k \kappa^4),
 \label{error-exp-kappa}
\end{eqnarray}
where $W_T \sim N(0,T)$ and $\mu$ and $\sigma$ are functions of the parameters 
determined by $e_n^{(0)}$ and $e_n^{(1)}$ above, and
which are bounded for fixed $\lambda>0$.

%

For the large wavenumber regime $\kappa > h^{-m}$, for any $m>0$, 
a similar calculation to the one in Section \ref{subsec:stab} shows that,
under the mean-square stability condition (\ref{theta-sig-stab}), 
\[
X_N = o(k^p) \quad \forall \, p>0.
\]
Following the analysis in \cite{carter} for the deterministic case ($\rho = 0$ in the present setting),
using discrete and continuous Fourier pairs,
\begin{eqnarray*}
V_j^n &=& \fracs{1}{2\pi} \int_{-\pi/h}^{\pi/h} X_n(\kappa) \exp(i\kappa h j) \ {\rm d}\kappa, \\
v(t_n, j h) &=& \fracs{1}{2\pi} \int_{-\infty}^{\infty} X(t_n,\kappa) \exp(i\kappa h j) \ {\rm d}\kappa,
\end{eqnarray*}
we can use the decay of $X_n$ and $X$ for large $\kappa$ to deduce
\begin{eqnarray*}
V_j^N - v(T,jh) &=& 
\fracs{1}{2\pi} \int_{-\pi/h}^{\pi/h} (X_n(\kappa)-X(T,\kappa)) \exp(i\kappa h j) \ {\rm d}\kappa + o(k) \\
&=&
\fracs{1}{2\pi} \int_{-h^{-m}}^{h^{-m}} X(T,\kappa) \left(\exp(S_N)-1\right) \exp(i\kappa h j) \ {\rm d}\kappa + 
o(k),
\end{eqnarray*}
and together with (\ref{error-exp-kappa}) we obtain the result upon expanding $\exp$ and integrating.
\end{proof}

%
%

\section{Convergence tests}
\label{sec:firsttests}

We now test the accuracy and stability of the scheme numerically.

The computations were conducted with the following set of  parameters for (\ref{spde}),
taken from \cite{bush}:
$\rho = 0.2$, $\mu=0.081$.

As initial data, we use $v(0,x)\!=\!\delta(x\!-\!x_0)$ with $x_0 = 5$.
In this case, the exact solution to the SPDE can be seen to be
\begin{equation}
v(T,x) = \frac{1}{\sqrt{2\pi \, (1\!-\!\rho)\,T}}\ 
\exp\left( -\ \frac{(\rule{0in}{0.16in}x - x_0 - \mu\, T - \sqrt{\rho} \, M_T)^2}{2\,(1\!-\!\rho)\,T} \right).
\label{eq:gaussian}
\end{equation}

For the computations, we localise the range of $x$ values to [-16/3,16] and set homogeneous Dirichlet boundary conditions.
This has been seen to introduce negligible localisation error numerically for the above parameters.

The operators $D_1$ and $D_2$ in (\ref{discreteimpl}) now have to be interpreted as finite difference matrices including the boundary conditions
for the first and last element.
In every timestep, the scheme requires the solution of a tridiagonal linear system similar to the one for the heat equation,
and therefore has the same computational complexity (linear in the number of grid points) as the explicit scheme.



The analytical solution $u$ allows us to approximate the mean-square $L_2$-error
using $J$ mesh intervals and $N$ timesteps by
\begin{eqnarray}
\nonumber
E(h,k)^2 &=& \EE\left[\sum_{j=0}^J (v_j^N(\omega)-v(N k,j h;\omega))^2 \,  h\right] \\
&\approx& \frac{1}{M} \sum_{m=1}^M
\sum_{j=0}^J (v_j^N(\omega_m)-v(N k,j h;\omega_m))^2 \, h,
\label{errex}
\end{eqnarray}
where the expectation over Brownian paths $\omega$ is approximated by the average over
$M$ samples $\omega_m$.

Anticipating applications where the exact solution is unknown, we also define error measures based on a  fine grid 
solution $f$ with mesh parameters $k$ and $h$, 
and a coarse solution $c$ with mesh parameters $4 k$ and $2 h$,
\begin{eqnarray}
\nonumber
e(h,k)^2 &=& \EE\left[\sum_{j=0}^{J/2} (f_{2j}^{N}(\omega)-c_j^{N/4}(\omega))^2 \, h\right] \\
&\approx& \frac{1}{M} \sum_{m=1}^M
\sum_{j=0}^{J/2} (f_{2j}^N(\omega_m)-c_j^{N/4}(\omega_m))^2 \, h.
\label{errest}
\end{eqnarray}

Iterating the refinement, we get a sequence of decreasing grid sizes 
$h_l = h_0\, 2^{-l}$ and timesteps
$k_l = k_0\, 4^{-l}$, and denote $E_l = E(h_l,k_l)$ the mean-square $L_2$-error
at level $l\ge 0$ and $e_l = e(h_l,k_l)$.
In the following example, $h_0=4/3$, $k_0=1/4$.
The refinement factors are determined by the stability constraint of the explicit scheme for $k/h^2$ and
the $O(k,h^2)$ convergence order.

Note that $x_0$ does not coincide with a grid point. We 
apply the Dirac initial data to a basis of hat functions to
retain second order convergence, see \cite{pooley1}.

Fig.~\ref{fig:unbddplots} shows the computed values of $E_l^2$ and $e_l^2$ for
the explicit scheme, $\theta=\sigma=0$, the drift implicit scheme $\theta=1$, $\sigma=0$,
and the `Milstein-anti-implicit' Crank-Nicolson scheme, $\theta=0.5$, $\sigma=-1$.
The choice of $h_l$ and $k_l$ is within the stability region of all schemes.
\begin{figure}
 \psfrag{explest}[l][l][0.4]{$e_l^2$, expl.}
 \psfrag{implest}[l][l][0.4]{$e_l^2$, impl.}
 \psfrag{cnest}[l][l][0.4]{$e_l^2$, C.N.}
 \psfrag{explex}[l][l][0.4]{$E_l^2$, expl.}
 \psfrag{implex}[l][l][0.4]{$E_l^2$, impl.}
 \psfrag{cnex}[l][l][0.4]{$E_l^2$, C.N.}
\psfrag{variance}[l][l][0.7]{$\log_2 e_l^2$, $\log_2 E_l^2$}
\psfrag{level}[l][l][0.7]{$l$}
\begin{center}
\includegraphics[width=.7\textwidth]{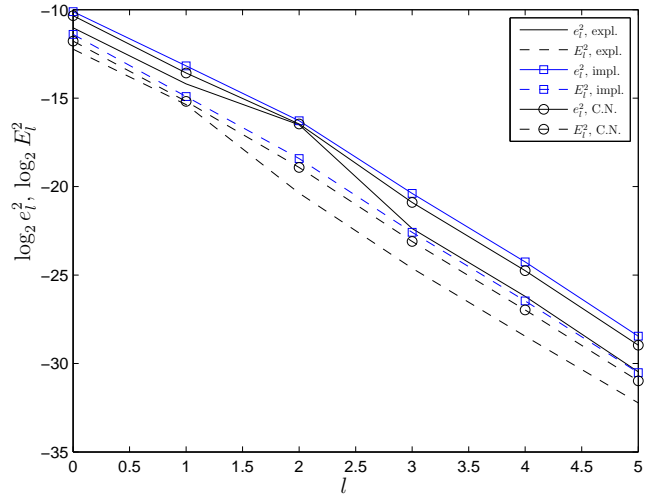}
\end{center}
\caption{
Mean-square error measures $E_l$ and $e_l$ (as explained in the text) for the explicit (`expl.', $\theta=0$, $\sigma=0$), drift implicit (`impl.', $\theta=1$, $\sigma=0$) and Crank-Nicolson-type (`C.N.', $\theta=0.5$, $\sigma=-1$) Milstein schemes.
}
\label{fig:unbddplots}
\end{figure}
The results confirm the theoretical $O(k,h^2)$ convergence, and show that the errors are very similar indeed for all schemes.



\section{An application with locally refined meshes}
\label{sec:results}

\subsection{An initial-boundary value problem}

In this section, we consider the initial-boundary value problem on the positive half-line,
\begin{eqnarray}
\label{spde1}
\dv &=& -\mu\, \frac{\partial v}{\partial x} \dt 
+ \frac{1}{2} \, \frac{\partial^2  v }{\partial x^2} \dt
- \sqrt{\rho}\, \frac{\partial v}{\partial x} \dM_t, \\
v(0,\cdot) &=& \delta(\cdot-x_0), 
\label{ic1}
\\
v(\cdot,0) &=& 0,
\label{bc1}
\end{eqnarray}
i.e., with an absorbing boundary condition at 0.

We use the same data as in Section \ref{sec:firsttests}, and for the numerical tests solve on $[0,16]$ to approximate the positive half-line.
The value of 16 was chosen large enough that truncation experimentally had negligible impact on the results. To illustrate, 
for $T=5$, the standard deviation of each $X_t^i$ in (\ref{iproc}) is $\sqrt{5} \approx 2.2$, so 16 is approximately 5 standard deviations away
from their starting point $X_0^i=5$. In contrast, the absorbing boundary at 0 is just over 2 standard deviations away, which suggests the
fraction of absorbed particles (lost mass of $v$) should be in the order of magnitude of $5\%$.

\begin{figure}
 \psfrag{explest}[l][l][0.4]{expl., unb.}
 \psfrag{implest}[l][l][0.4]{impl., unb.}
 \psfrag{cnest}[l][l][0.4]{C.N., unb.}
 \psfrag{explest2}[l][l][0.4]{expl., bdd}
 \psfrag{implest2}[l][l][0.4]{impl., bdd}
 \psfrag{cnest2}[l][l][0.4]{C.N., bdd}
\psfrag{variance}[l][l][0.7]{$\log_2 e_l^2$, $\log_2 E_l^2$}
\psfrag{level}[l][l][0.7]{$l$}
\begin{center}
\includegraphics[width=.7\textwidth]{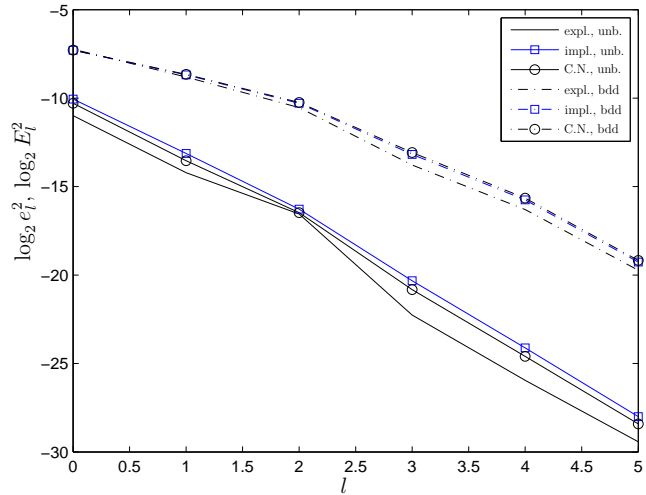}
\end{center}
\caption{
Mean-square error measure $e_l$ (as explained in the text)
for the explicit (`expl.', $\theta=0$, $\sigma=0$), drift implicit (`impl.', $\theta=1$, $\sigma=0$) and Crank-Nicolson-type (`C.N.', $\theta=0.5$, $\sigma=-1$) Milstein schemes,
for the bounded case in comparison with the unbounded case already seen in 
Fig.~\ref{fig:unbddplots}.
}
\label{fig:unbddandbddplots}
\end{figure}

Fig.~\ref{fig:unbddandbddplots} shows that although the estimated error is still asymptotically of the same order in this case,
it is substantially larger.
It is known from \cite{kl98} that the solution on the half-line is only in $H^1$ in space but does not have $L_2$ second derivative, however
$x u_{xx} \in L_2$.

\subsection{Local mesh refinement}

To remove the singularity of the computed solution at $x=0$,
one might introduce local coordinate stretching, 
i.e., a new coordinate $y$ and increasing one-to-one
function $f: [0,\infty)\rightarrow [0,\infty)$ with inverse $g$
such that
\[
v(t,x) = w(t,f(x)) \quad \Leftrightarrow \quad v(t,g(y)) = w(t,y), \quad \forall t,x,y\ge 0.
\]
The SPDE (\ref{spde}) in $y$-coordinates reads
\begin{equation}
\label{spde-trans}
\dw = \left(-\mu\, f' \!\! \circ \! g \, + \frac{1}{2} \, f'' \!\! \circ \! g \right) \, \frac{\partial w}{\partial y} \dt 
+ \frac{1}{2} \, (f' \! \circ g)^2 \, \frac{\partial^2  w }{\partial y^2} \dt
- \sqrt{\rho}\, f' \! \circ  g \, \frac{\partial w}{\partial y} \dM_t,
\end{equation}
and the Milstein finite difference schemes are defined accordingly.

Conversely, this is closely related to a discretisation of the original SPDE on a non-uniform mesh with nodes $x_n = g(n h)$.

A distinct choice of transformation is $y=\sqrt{x}$, because then
\[
\frac{\partial^2  w }{\partial y^2} = 4 x \frac{\partial^2  v}{\partial x^2} + 2 \frac{\partial v}{\partial x},
\]
and from \cite{krylov94} the right-hand-side is known to be square-integrable in $x$.
This does not imply, however, that $w \in H^2$ and does not lend itself easily to an improved numerical analysis.

We now investigate the numerical improvement in accuracy for a specific application.

\subsection{Application to credit derivatives}

\cite{bush} show how the equations (\ref{spde1}) to (\ref{bc1}) can be used to model credit baskets:
there, (\ref{spde1}) describes the evolution of a firm value distribution of a large basket of defaultable obligors,
where each firm value follows (\ref{iproc});
$x_0$ in (\ref{ic1}) is the firm value at the initial time;
(\ref{bc1}) models default of a firm when its value process crosses a default threshold at $x=0$.
The basket loss, i.e.\ the fraction of firms that have not survived, is then given by
\begin{equation}
L_t = 1-\int_0^\infty u(t,x) \dx.
\end{equation}

The loss model can be used as the basis for the valuation of basket credit derivatives,
which are structured by a sequence of regular fee payments by the protection buyer, in
return for a payment by the protection seller if a certain default event occurs.
A standardised such product is a collateralized debt obligation (CDO) where the payments
depend on the losses in a certain segment of the basket, measured by attachment points $a$
and detachment points $d$, over a certain time horizon $T$.
The outstanding tranche notional is then defined as
\begin{equation}
Z_t = \max(d-L_t,0) - \max(a-L_t,0).
\end{equation}
We will consider a maturity  $T=5$ and a single tranche $[a,d] = [0,0.03]$.

A survey of products and models can be found e.g.\ in \cite{schoenbucher},
a derivation of pricing formulae in the present model in \cite{bush}.

The main quantities that enter the formulae for tranche spreads are expected, discounted (with interest rate $r$, here 0.042)
spread payments,
so we will be considering here
\[
P = \sum_{i=1}^{n} \e^{-r T_{i}} \mathbb{E}[{Z}_{T_{i-1}} - {Z}_{T_{i}}],
\]
where $T_i = i q$, $q=0.25$ (quarterly payments), $n=20$. 

We now give results with and without coordinate transformation.
For the sake of completeness, we provide the approximated loss function in transformed coordinates
\[
L_t = 1-\int_0^\infty w(t,y) g'(y) \dy \approx 1 - h \sum_{j=1}^{J-1} w(t,y_j) g'(y_j),
\]
where we use the last expression as numerical approximation to the losses.

Let ${P}_{l}$ be an approximation to $P$ with mesh size $h_l = h_0 \, 2^{-l}$ and time step
$k_l = k_0 \, 4^{-l}$ for $l>0$ and $h_0=8/5$, $k_0=1/4$.
We use the $\theta$-$\sigma$ scheme with $\theta=0.5$ and $\sigma=-1$ for its unconditional stability.
We compute estimators $\widehat{Y}_l$ to $\EE[P_l-P_{l-1}]$ for $l>0$ in order to estimate the contributions
of individual refinement levels. We do this by averaging $P_l-P_{l-1}$ over $N_l$ sample paths $(M_t)_{t\in [0,T]}$ (identical for $P_l$ and $P_{l-1}$,
but independent for different $\widehat{Y}_l$).
We show $\EE[\widehat{Y}_l]$ and $V_l = N_l \VV[\widehat{Y}_l]$ in Fig.~\ref{fig:scalingplots},
where the number of samples, $N_l$, is chosen to make the simulation error negligible.
\begin{figure}
\begin{center}
\psfrag{alpha=1}[l][l][0.4]{$\alpha=1$}
\psfrag{alpha=1/2}[l][l][0.4]{$\alpha=0.5$}
\psfrag{level l}[l][l][0.4]{$l$}
\psfrag{log2 variance}[l][l][0.4]{$V_l$}
\includegraphics[width=.45\textwidth]{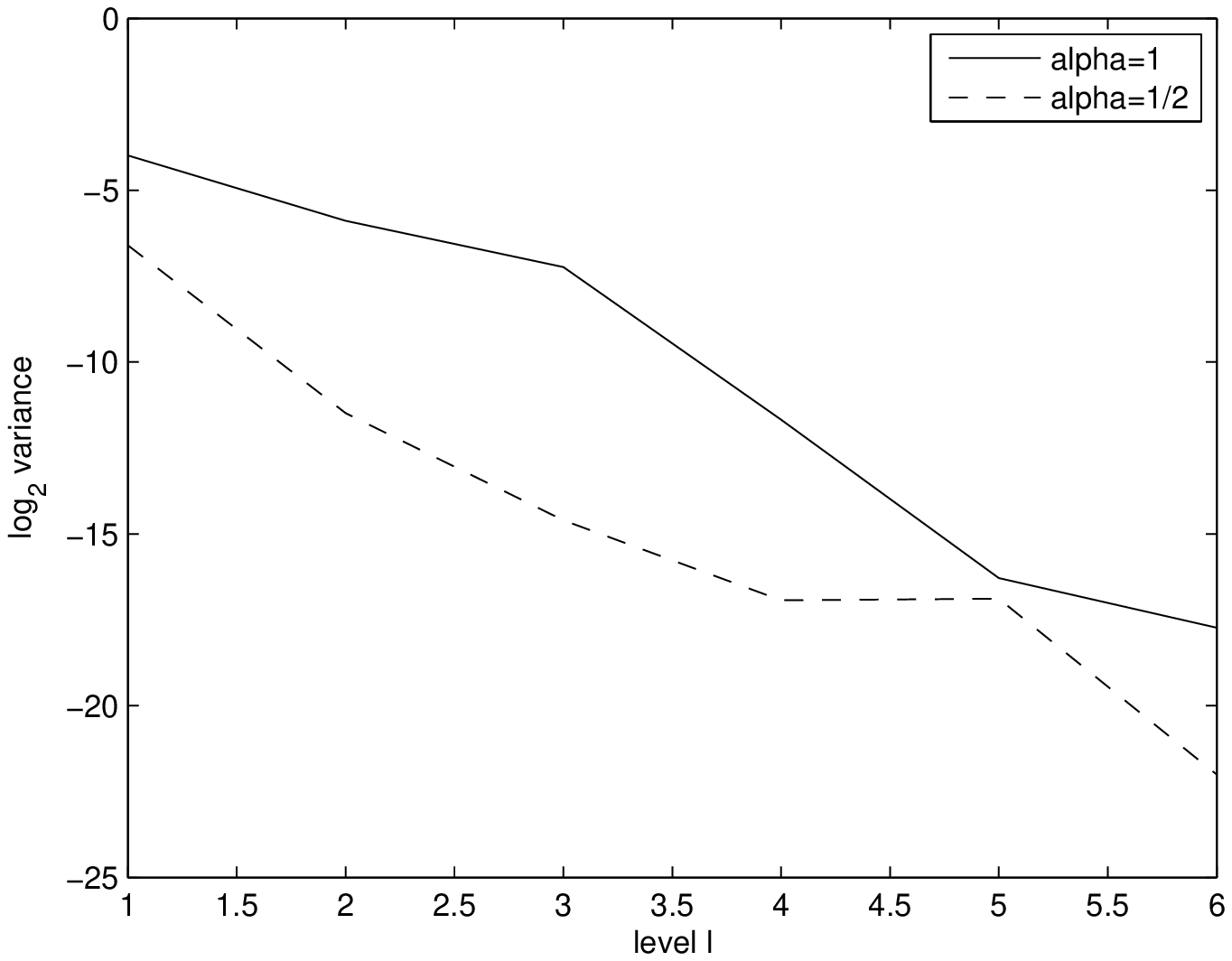}
\hfill
\includegraphics[width=.45\textwidth]{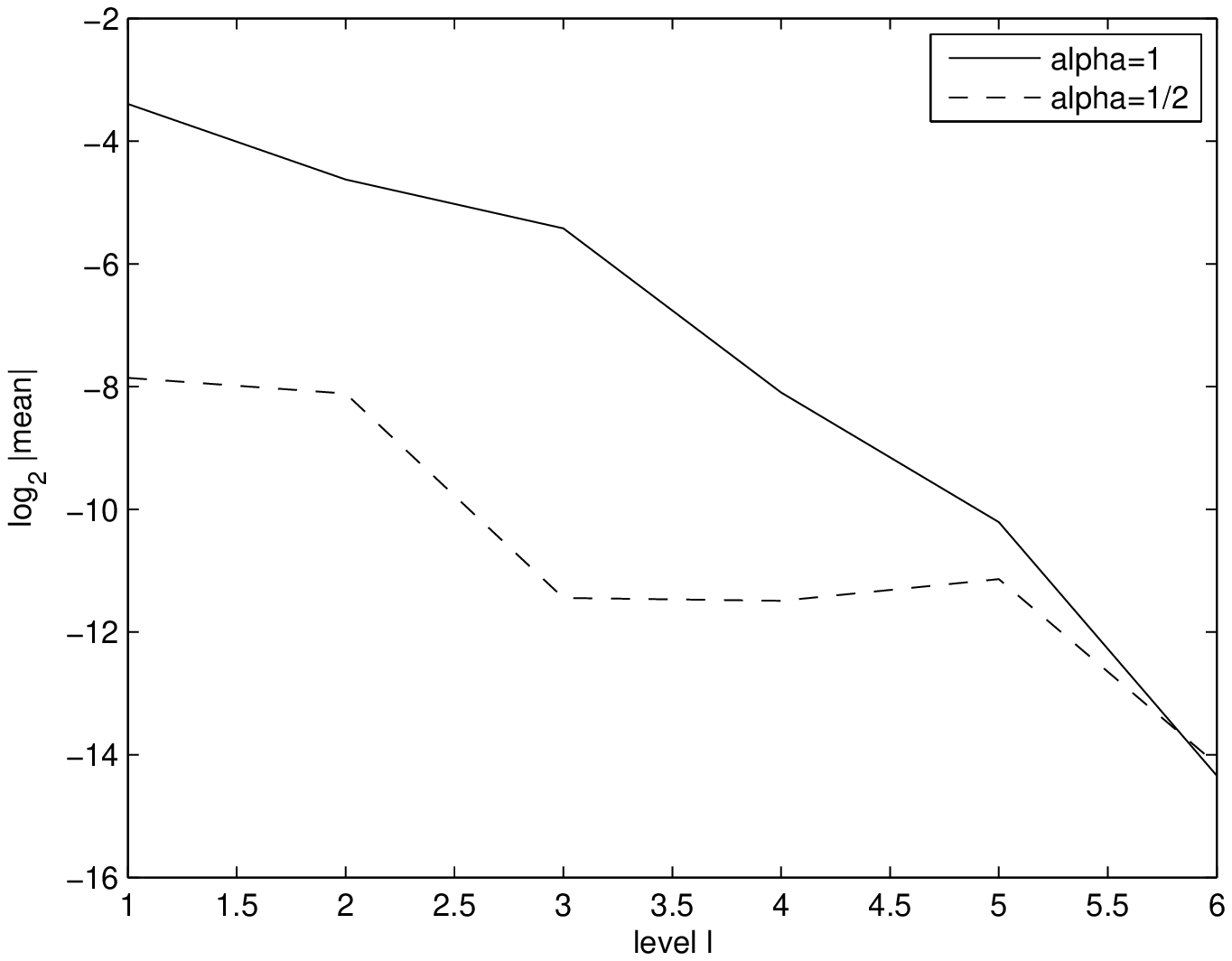}
\end{center}
\caption{
Logarithm of variance $V_l$
and mean $\EE[P_{l} - P_{l-1}]$
of the correction from refinement level $l-1$ to level $l$, for a power grid streching $x^\alpha$.
Shown is the effect of the change from the original coordinates (equivalent to the case $\alpha=1$)
to the square root stretching (for which $\alpha=1/2$).
}
\label{fig:scalingplots}
\end{figure}

The numerical results show that on coarse levels, both the mean and the variance of the estimators are much smaller
when the computation is done in $y$ coordinates with $y=\sqrt{x}$ instead of the original $x$ coordinates.
The smaller variance can be exploited by writing, in the spirit of the multi-level Monte Carlo method of
\cite{giles1, gr11},
\begin{eqnarray*}
\EE[P_L] &=& \EE\left[\sum_{l=0}^L \widehat{Y}_l\right], \\
\VV\left[\sum_{l=0}^L \widehat{Y}_l\right] &=& \sum_{l=0}^L \frac{V_l}{N_l},
\end{eqnarray*}
for $\widehat{Y}_l$ as above for $l>0$ and an estimator $\widehat{Y}_0$ to $\EE[P_0]$, and where $V_l$ is the variance of $\widehat{Y}_l$ for $N_l=1$.

Using the sum of $\widehat{Y}_l$ as multilevel estimator for $\EE[P]$, one can optimise $N_l$ to give minimal overall computational complexity for a given combined variance.
As $V_l$ has been reduced on coarser levels due to the grid stretching,
$N_l$ can be smaller there compared to the uniform grid.
Given the complexity will be largely determined by the number of samples on the coarsest levels (see \cite{gr11}), a decrease
in the variance of around 100 on those levels (see Fig.~\ref{fig:scalingplots}) allows the number of paths to decrease by a factor of 100,
which gives significant computational savings.

\section{Conclusions}
\label{sec:conclusions}

We consider implicit variants of the Milstein scheme for a class of SPDEs, and show improved stability properties.
In particular
we find that an `anti'-implicit discretisation of the deterministic part of the Milstein correction leads to an unconditionally stable scheme.
This is of some importance for stiff systems arising from the SPDE discretisation, especially for locally refined meshes, where noticable
computational savings are observed.

An important open question is a complete analysis of the numerical approximation of initial-boundary value problems for the considered
SPDE.
It might also be interesting to investigate similar ideas in the context of higher order expansions of the stochastic integral.

{
\small
\begin{center}%
{\bfseries Acknowledgements}
\vspace{-.5em}%
\end{center}

The author would like to thank Lukas Szpruch for helpful discussions on implicit Milstein schemes and Mike Giles for discussions
on Fourier analysis in this context.
}



\end{document}